\documentclass[english,10pt]{amsart}

\usepackage{amsmath, amssymb, amsthm, enumerate}
\usepackage[all]{xy}
\usepackage[activeacute]{babel}

\newtheorem{thm}[equation]{Theorem}
\makeatletter
\let\c@subsubsection\c@equation
\makeatother

\newtheorem{prop}[equation]{Proposition}
\newtheorem{cor}[equation]{Corollary}

\theoremstyle{remark}
\newtheorem{rmk}[equation]{Remark}

\theoremstyle{definition}
\newtheorem{defi}[equation]{Definition}

\newcommand{\Hom}{\mathrm{Hom}}

\newcommand{\Gm}{\mathbb G _{m}}

\newcommand{\spec}[1]{\mathrm{Spec}(#1)}
\newcommand{\generators}{\mathcal G}

\newcommand{\hr}{\sphere _{R}}

\newcommand{\sphere}{\mathbf 1}

\DeclareMathOperator*{\hocolim}{hocolim}
\newcommand{\colim}{\mathrm{colim}}

\newcommand{\DMeff}{DM^{\mathrm{eff}}}

\newcommand{\northogonal}[1]{DM ^{\perp}(#1)}
\newcommand{\northogonalSH}[1]{\mathcal{SH} ^{\perp}(#1)}

\numberwithin{equation}{subsection}

\begin{document}


\title{On the convergence of the orthogonal spectral sequence}


\author{Cesar Galindo}
\address{Instituto de Matem\'aticas, Ciudad Universitaria, UNAM, DF 04510, M\'exico}
\email{cesar\_gal@ciencias.unam.mx}

\author{Pablo Pelaez}
\address{Instituto de Matem\'aticas, Ciudad Universitaria, UNAM, DF 04510, M\'exico}
\email{pablo.pelaez@im.unam.mx}





\begin{abstract}
We  show that the orthogonal spectral sequence 
introduced by the second author is strongly
convergent in Voevodsky's triangulated category of motives $DM$
over a field $k$. In the context
of the Morel-Voevodsky $\mathbb A ^1$-stable homotopy category
we provide concrete examples where the spectral sequence
is not strongly convergent, and give a criterion under which the
strong convergence still holds.  This criterion holds
for Voevodsky's slices, and as a consequence we obtain
a spectral sequence which converges strongly to the $E_1$-term
of Voevodsky's slice spectral sequence.
\end{abstract}


\maketitle

\section{Introduction}  \label{sec.introd}

\subsection{}
In order to study finite filtrations on the Chow groups of a smooth
projective variety $Y$ over a field $k$ which satisfy some of the properties
of the still conjectural Bloch-Beilinson-Murre filtration
\cite{MR923131}, \cite{MR558224}, \cite{MR1225267}, 
the second author introduced a tower of triangulated functors
$bc_{\leq \bullet}$ \cite{MR3614974}:
\begin{equation} \label{dual.st}	
	\cdots \rightarrow bc_{\leq n-1} \rightarrow bc_{\leq n} 
	\rightarrow bc_{\leq n+1} \cdots \rightarrow 
\end{equation}
in Voevodsky's triangulated category of motives $DM$.  

The filtration
on the Chow groups
with coefficients in a commutative ring $R$
is defined by evaluating the tower \eqref{dual.st} in the motive of a point $\hr$
and then mapping $M(Y)(-q)[-2q]$ 
into $bc_{\leq \bullet}(\hr)$, where $M(Y)$ is the motive
of $Y$, and $(-q)$ (resp. [-2q]) is defined in terms of the Tate twist
(resp. suspension) in $DM$, see \ref{def.DM}.  
This process gives a filtration in the Chow groups since
\cite{MR1883180}: 
\[CH^{q}(Y)_{R}\cong \Hom _{DM}(M(Y)(-q)[-2q], \sphere _{R}).\]

Given $A, B\in DM$, one may as well evaluate the tower \eqref{dual.st}
in $A$ and then map $B$ into $bc_{\leq \bullet}(A)$.  Then one
obtains a spectral sequence:
\begin{equation}  \label{eq.conv.ss}
E^{1}_{p,q}=\Hom_{DM}(B, (bc_{p/p-1}A)[q-p])
\Rightarrow \Hom_{DM}(B, A)
\end{equation}
where $bc_{p/p-1}A$ is defined in terms of a canonical distinguished
triangle in $DM$ \cite[3.2.8]{MR3614974}: 
\[bc_{\leq p-1}A\rightarrow bc_{\leq p}A\rightarrow bc_{p/p-1}A
\]

The  goal of this paper is to study the convergence properties of
the spectral sequence \eqref{eq.conv.ss}.

Our main result \eqref{main.thm} shows that 
the spectral sequence \eqref{eq.conv.ss} is strongly convergent
for $B=M(X)(r)[s]$, with $r, s\in \mathbb Z$ and $X$ an arbitrary
smooth scheme of finite type over $k$.
We show as well that the analogous result does not hold
in the
Morel-Voevodsky $\mathbb A ^1$-stable homotopy category 
$\mathcal {SH}$ by 
providing explicit counterexamples 
\eqref{prop.1notconv}-\eqref{cor.notconv}.  On the other hand,
we verify that under suitable conditions for $A\in \mathcal {SH}$ 
the spectral sequence is strongly convergent \eqref{prop.orth.conv}, and that
this conditions are satisfied for Voevodsky's slices $s_mG\in \mathcal{SH}$.
As a direct consequence we obtain
a spectral sequence converging strongly to the $E_1$-term of
Voevodsky's slice spectral sequence \eqref{cor.Voev.slice}.

In a future work, we will apply the spectral sequence \eqref{eq.conv.ss}
and its convergence properties to describe
the higher terms for the filtration \cite[6.14]{MR3614974}
mentioned above
 on the Chow groups of
smooth projective varieties with rational coefficients.

We refer the reader to \cite{kohrita2019filtrations} and \cite{bondarko2021smooth} for filtrations which
are constructed by a related process.

\section{Preliminaries}

In this section we fix the notation that will be used throughout the rest
of the paper and collect together facts  from the literature
that will be necessary to establish our results.  

\subsection{Definitions and Notation}	\label{subsec.defandnots}		

We  fix a base field $k$.
We will write $Sch_{k}$ for the category of $k$-schemes
of finite type and $Sm_{k}$ for the full
subcategory of $Sch_{k}$ consisting of smooth $k$-schemes regarded
as a site with the Nisnevich topology.	  

We will use the following notation in all the categories under consideration: $0$ will
denote the zero object, and $\cong$ will denote that a map (resp. a functor) is an
isomorphism (resp. an equivalence of categories).  

We shall use freely the language of triangulated categories.  Our main reference will 
be \cite{MR1812507}.  Given a triangulated category, we will write $[1]$ 
(resp. $[-1]$) to denote its suspension 
(resp. desuspension) functor; and for $n>0$, $[n]$ (resp. $[-n]$)
will be the composition of $[1]$
(resp. $[-1]$) iterated $n$-times.  If $n=0$,
$[0]$ will be the identity functor.  Given an inductive system
$\cdots T_n\rightarrow T_{n+1}\rightarrow \cdots$, its homotopy
colimit, $\hocolim _{n\rightarrow \infty}T_n$ will be defined as
in \cite{MR1812507}.

\subsection{Triangulated categories} \label{subsec.somefacts}

Let $\mathcal T$ be a compactly generated triangulated category in the sense of Neeman
\cite[Def. 1.7]{MR1308405} with set of compact generators  $\mathcal G$.  For 
$\mathcal G '\subseteq \mathcal G$, let $Loc(\mathcal G ')$ denote the smallest full
triangulated subcategory of $\mathcal T$ which contains $\mathcal G'$ and is closed
under arbitrary (infinite) coproducts.

\begin{defi}  \label{def.orthn}
Let $\mathcal T '\subseteq \mathcal T$  be a triangulated 
subcategory.		We will write $\mathcal T '^{\perp}$ for the full subcategory of
$\mathcal T$ consisting of the objects $E\in \mathcal T$ such that
for every $K\in \mathcal T '$:  $\Hom _{\mathcal T}(K,E)=0$.

If $\mathcal T '=Loc(\mathcal G')$ and $E\in \mathcal T '^{\perp}$,
we will say that $E$ is \emph{$\mathcal G'$-orthogonal}.
\end{defi} 

\subsection{Slice and orthogonal towers}  \label{subsec.slcos}
As in \eqref{subsec.somefacts}.
Consider a family of subsets of $\generators$: $\mathcal S=\{ \generators _{n} \} _{n
\in \mathbb Z}$
such that $\generators _{n+1}\subseteq \generators _{n}\subseteq \generators$ for
every $n\in \mathbb Z$.

Thus, we obtain a tower of full triangulated subcategories of $\mathcal T$:
\begin{equation}  \label{eq.genslitow}
\cdots \subseteq Loc(\generators _{n+1}) \subseteq Loc(\generators _{n})
\subseteq Loc(\generators _{n-1}) \subseteq \cdots
\end{equation}

We will call \eqref{eq.genslitow} the \emph{slice tower} determined by $\mathcal S$.  
The reason
for this terminology is \cite{MR1977582}, \cite{MR2249535}, \cite[p. 18]{MR2600283}.  
If we consider the orthogonal categories
$Loc(\generators _{n})^\perp$ \eqref{def.orthn}, we obtain a tower of full triangulated
subcategories of $\mathcal T$:
\begin{equation}  \label{eq.genslitow2}
\cdots \subseteq Loc(\generators _{n-1})^{\perp} \subseteq Loc(\generators _{n})^{\perp}
\subseteq Loc(\generators _{n+1})^{\perp} \subseteq \cdots
\end{equation} 
	
\subsubsection{Orthogonal covers} \label{def.birat.cover}		
	
Recall \cite[2.1.7(3)]{MR3614974} that
the inclusion,
		$ j_{n}: Loc(\generators _{n})^{\perp} \rightarrow \mathcal T$
	admits a right adjoint:
		\[p_{n}:\mathcal T \rightarrow Loc(\generators _{n})^{\perp},\]
	which is also a triangulated functor.		
We define $bc_{\leq n}=j_{n+1}\circ p_{n+1}$.  

	\subsubsection{}  \label{s.counit-properties}
The counit $bc_{\leq n}=j_{n+1}p_{n+1}
\stackrel{\theta_{n}}{\rightarrow} id$ of the adjunction
in \eqref{def.birat.cover}
satisfies the following universal property (by
an argument parallel to \cite[3.2.4]{MR3614974}): 
	
	For any $A$ in $\mathcal T$ and for any 
	$B\in Loc(\generators _{n+1})^{\perp}$,
	the map $\theta ^{A}_{n}: bc_{\leq n}A \rightarrow A$
	in $\mathcal T$ induces an isomorphism of
	abelian groups:
		\[\xymatrix{\Hom _{\mathcal T}(B, bc_{\leq n}A) 
		\ar[r]_-{\cong}^-{\theta ^{A}_{n\ast}}
		& 
			\Hom _{\mathcal T}(B, A) }
		\]

	\subsubsection{}  \label{subsec.bc.idemp}
Observe that by construction $bc_{\leq n}A$ is in 
$Loc(\generators _{n+1})^{\perp}$ 
\eqref{def.birat.cover} and in addition
$Loc(\generators _{n+1})^{\perp}\subseteq Loc(\generators _{n+2})^{\perp}$ \eqref{eq.genslitow2}.  Thus,
it follows from \eqref{s.counit-properties} that there
exists  a canonical natural transformation $bc_{\leq n}
\rightarrow bc_{\leq n+1}$ and that
$bc_{\leq n} \circ bc_{\leq n+1}\cong bc_{\leq n}$  
\cite[3.2.6]{MR3614974}
(the argument works for any compactly generated triangulated category).
Hence, for every $A$ in $\mathcal T$ there is a functorial tower in 
$\mathcal T$
\cite[3.2.14, 3.2.15]{MR3614974}:

\begin{align}	\label{A.birtow}
\begin{array}{c}
	\xymatrix@C=1.5pc{\cdots \ar[r]
					& \ar[dr]|{\theta _{n}^{A}} bc_{\leq n}A
					\ar[r] & \ar[d]|{\theta _{n+1}^{A}} 
					bc_{\leq n+1}A \ar[r] 
					&  \cdots  \ar[r]&
					\hocolim_{n\rightarrow \infty} bc_{\leq n}A \ar[dll]^-{c}\\						
					&&  A &&}
\end{array}
\end{align}
where all the triangles commute.
We will call \eqref{A.birtow} the orthogonal tower of $A$.

	\begin{defi} \label{def.abut.fil}
For $A$, $B$ in $\mathcal T$ we consider the increasing filtration
$F_\bullet$ on $\Hom _{\mathcal T}(B,A)$ (resp. 
$\Hom _{\mathcal T}(B, \hocolim_{n\rightarrow \infty} bc_{\leq n}A)$),
where $F_p$ is given
by the image of 
	\begin{align*}
\theta _{p\ast}^{A} &:\Hom _{\mathcal T}(B, bc_{\leq p}A) 
\rightarrow \Hom _{\mathcal T}(B, A)\\
\mathrm{(resp.}\; \;
\lambda _{p, \ast}^A &:\Hom _{\mathcal T}(B, bc_{\leq p}A)\rightarrow
\Hom _{\mathcal T}
(B, \hocolim_{n\rightarrow \infty} bc_{\leq n}A) \mathrm{)}	
	\end{align*}
where $\lambda _p^A: bc_{\leq p}
A \rightarrow \hocolim_{n\rightarrow \infty} bc_{\leq n}A
$ is the canonical map into the homotopy colimit
\eqref{A.birtow}.
	\end{defi}

	\subsection{The orthogonal spectral sequence}  \label{thm.birspecseq}
Let $A, B$ be in $\mathcal T$.  By an argument parallel to
 \cite[Thm. 3.2.8]{MR3614974} there exist
canonical triangulated functors $bc_{p/p-1}:\mathcal T \rightarrow 
\mathcal T$,
$p \in \mathbb Z$ which fit in a
natural distinguished triangle in $\mathcal T$:
	\begin{align} \label{dist.triang.ss}
bc_{\leq p-1}A\rightarrow bc_{\leq p}A\rightarrow bc_{p/p-1}A
	\end{align}	

 Then
\eqref{A.birtow} induces a spectral sequence of homological type
\cite[Thm. 3.2.16]{MR3614974}:
	\begin{align} \label{gen.spec.seq}
E^{1}_{p,q}=\Hom_{\mathcal T}(B, (bc_{p/p-1}A)[q-p])
\Rightarrow \Hom _{\mathcal T}(B,A)
	\end{align}
with differentials $d_r:E^{r}_{p,q}\rightarrow E^{r}_{p-r,q-r+1}$
and where the abutment is given by the associated graded 
group for the increasing filtration \eqref{def.abut.fil}
$F_{\bullet}$ of $\Hom _{\mathcal T} (B,A)$.
	 
Similarly, 
the horizontal row in \eqref{A.birtow}
induces a spectral sequence of homological type:
	\begin{align} \label{hocolim.spec.seq}
E^{1}_{p,q}=\Hom
_{\mathcal T}(B, (bc_{p/p-1}A)[q-p])
\Rightarrow \Hom _{\mathcal T}(B,\hocolim_{n\rightarrow \infty} bc_{\leq n}A)
	\end{align}
with exactly the same
differentials as \eqref{gen.spec.seq}
and where the abutment is given by the associated graded 
group for the increasing filtration \eqref{def.abut.fil}
$F_{\bullet}$ of \linebreak
$\Hom _{\mathcal T} (B,\hocolim\limits_{n\rightarrow \infty} bc_{\leq n}A)$.

Now, we observe that the map $c$ in \eqref{A.birtow}
induces a map of spectral sequences
$\eqref{hocolim.spec.seq} \rightarrow \eqref{gen.spec.seq}$ which is 
the identity on the
$E_1$-terms:

	\begin{align} \label{comp.spec.seq}
		\begin{array}{r}
\xymatrix@C=0.5pt{E^{1}_{p,q}=\Hom_{\mathcal T}(B, (bc_{p/p-1}A)[q-p]) 
	\ar@{=>}[rrr] \ar@{=}[d] &&&
	\Hom _{\mathcal T}(B, \hocolim\limits_{n\rightarrow \infty}
	bc_{\leq n}A) \ar[d]^-{c_\ast} \\
	E^{1}_{p,q}=\Hom_{\mathcal T}(B, (bc_{p/p-1}A)[q-p]) 
	\ar@{=>}[rrr] &&& \Hom _{\mathcal T}(B,A)}
		\end{array}
	\end{align}
	
\subsection{Voevodsky's triangulated category of motives} \label{def.DM}

We will only consider motives with $R$-coefficients, where
$R=\mathbb Z [\frac{1}{p}]$ and
$p$ is the exponential characteristic of the base field $k$.

Let $Cor _k$ be
the Suslin-Voevodsky category of finite correspondences over $k$,
i.e. the category with
the same objects as $Sm_{k}$
and morphisms $c(U,V)$ given by the $R$-module of finite
relative cycles on $U\times _{k}V$ over $U$ \cite{MR1764199}	 
with composition as in
\cite[p. 673 diagram (2.1)]{MR2804268}.  The graph of a morphism 
in $Sm_{k}$
induces a functor $\Gamma : Sm_{k}\rightarrow Cor_{k}$.  A Nisnevich sheaf with
transfers is an additive contravariant functor $\mathcal F$ from $Cor_{k}$ to the category
of $R$-modules 
such that the restriction $\mathcal F \circ \Gamma$ is a Nisnevich
sheaf.  Let $Shv^{tr}$ be the category of Nisnevich sheaves with
transfers which is an abelian category \cite[13.1]{MR2242284}.  
Given $X\in Sm_{k}$, we will write
$\mathbb Z _{tr}(X)$ for the Nisnevich sheaf with transfers represented by $X$
\cite[2.8 and 6.2]{MR2242284}.

Consider
the category of chain complexes (unbounded) on $Shv ^{tr}$,
$K(Shv ^{tr})$,  equipped
with the injective model structure \cite[Prop. 3.13]{MR1780498}, and let
$D(Shv ^{tr})$ be its homotopy category.  Let 
$K^{\mathbb A ^{1}}(Shv ^{tr})$ be
the left Bousfield localization \cite[3.3]{MR1944041} of $K(Shv ^{tr})$ with respect to the set
of maps
$\{ \mathbb Z _{tr}(X\times _{k}\mathbb A ^{1})[n]\rightarrow \mathbb Z _{tr}(X)[n]:
X\in Sm_{k}; n\in \mathbb Z\}$ induced by the projections $p:X\times _{k}\mathbb A ^{1}
\rightarrow X$. Voevodsky's  triangulated category of effective motives
$\DMeff$ is the homotopy category of $K^{\mathbb A ^{1}}(Shv ^{tr})$ \cite{MR1764202}.

Let $T\in K^{\mathbb A ^{1}}(Shv ^{tr})$ denote the chain complex of the from 
$\mathbb Z_{tr}(\Gm)[1]$ \cite[2.12]{MR2242284}, where $\Gm$ is the $k$-scheme
$\mathbb A ^1\backslash \{ 0\}$ pointed by $1$.  
We consider the category of symmetric
$T$-spectra on $K^{\mathbb A ^{1}}(Shv ^{tr})$,
$Spt_{T}(Shv ^{tr})$, equipped with the model structure
defined in \cite[8.7 and 8.11]{MR1860878}, \cite[Def. 4.3.29]{MR2438151}.  
Voevodsky's  triangulated category of  motives
$DM$ is the homotopy category of $Spt_{T}(Shv ^{tr})$ 
\cite{MR1764202}.  

We will write $M(X)$ for the image of $\mathbb Z _{tr}(X)\in D(Shv^{tr})$, $X\in Sm_{k}$
under the $\mathbb A ^{1}$-localization map $D(Shv ^{tr})\rightarrow \DMeff$.  Let
$\Sigma ^{\infty}:\DMeff \rightarrow DM$ be the suspension functor 
\cite[7.3]{MR1860878} (denoted by $F_0$ in \emph{loc.\!\,cit.}), 
we will abuse notation and simply write $E$ for $\Sigma
^{\infty}E$, $E\in \DMeff$.  Given a map $f:X\rightarrow Y$ in $Sm _k$, we will 
write $f:M(X)\rightarrow M(Y)$ for the map induced by $f$ in $DM$.

Notice that, $\DMeff$ and $DM$ are tensor triangulated categories 
\cite[Thm. 4.3.76 and Prop. 4.3.77]{MR2438151} with unit $\sphere = M(\spec{k})$. 
We will write $A(1)$ for $A\otimes \mathbb Z _{tr}(\Gm)[-1]$, $A\in DM$ and inductively
$A(n)=(A(n-1))(1)$, $n\geq 0$.  We observe that the functor $DM\rightarrow DM$,
$A\mapsto A(1)$ is an equivalence of categories \cite[8.10]{MR1860878}, 
\cite[Thm. 4.3.38]{MR2438151};
we will write $A\mapsto A(-1)$ for
its inverse, and inductively $A(-n)=(A(-n+1))(-1)$,
$n>0$.  By convention $A(0)=A$ for $A\in DM$.

\subsubsection{Generators}  \label{sub.gens}

It is well known that
$DM$ is a compactly generated triangulated category \eqref{subsec.somefacts} with
compact generators \cite[Thm. 4.5.67]{MR2438151}:
\begin{equation}  \label{eq.DMgens}
 \generators _{DM}=\{ M(X)(p): X\in Sm_{k}; p \in \mathbb Z\}.
\end{equation}

Let $\generators ^{\mathrm{eff}}\subseteq \generators _{DM}$ be the set consisting of compact objects of the form:
\begin{equation}  \label{eq.DMeffgens}
 \generators ^{\mathrm{eff}}=\{ M(X)(p): X\in Sm_{k}; p\geq 0\}.
\end{equation}

If $n\in \mathbb Z$, we will write $\generators ^{\mathrm{eff}}(n)\subseteq 
\generators _{DM}$ for the set consisting of compact objects of the form:
\begin{equation}  \label{eq.DMeffgenstw}
 \generators ^{\mathrm{eff}}(n)=\{ M(X)(p): X\in Sm_{k}; p\geq n\}.
\end{equation}

\subsubsection{}  \label{subsubsec.cancel}
By Voevodsky's cancellation theorem \cite{MR2804268}, 
the suspension functor 
$\Sigma ^{\infty}:\DMeff \rightarrow DM$ induces an equivalence of 
categories between 
$\DMeff$ and the full triangulated subcategory 
$Loc(\generators ^{\mathrm{eff}})$ of $DM$
\eqref{subsec.somefacts}.
We will abuse notation and write $\DMeff$ for 
$Loc(\generators ^{\mathrm{eff}})$.  Strictly speaking 
\cite[Cor. 4.10]{MR2804268} is only stated for perfect base fields,
but by the work of Suslin 
\cite[Cor. 4.13, Thm. 4.12 and Thm. 5.1]{MR3590347}
it follows that the result holds as well for non-perfect base fields.

\subsubsection{}  \label{subsubsec.neffDM}
We will write $\DMeff (n)$ for the full triangulated subcategory $Loc(\generators
^{\mathrm{eff}} (n))$ of $DM$ \eqref{subsec.somefacts}, and $\northogonal{n}$
for the orthogonal category $Loc(\generators^{\mathrm{eff}} (n))^{\perp}$
\eqref{def.orthn}.  Notice that $\DMeff (n)$ is compactly generated with set of generators
$\generators ^{\mathrm{eff}}(n)$ \cite[Thm. 2.1(2.1.1)]{MR1308405}.

\subsubsection{}  \label{subsec.birattow}  

We will consider \eqref{eq.genslitow}, \eqref{eq.genslitow2},
\eqref{A.birtow}, \eqref{def.abut.fil}, \eqref{gen.spec.seq},
\eqref{hocolim.spec.seq} and \eqref{comp.spec.seq} in $DM$
for the family $\mathcal S=\{ \generators ^{\mathrm{eff}}(n) \}_{n\in \mathbb Z}$
of subsets of $\generators _{DM}$ \eqref{eq.DMeffgenstw}.  	

\subsection{The Morel-Voevodsky $\mathbb A ^1$-stable homotopy
category}

We refer the reader to \cite[\S, Thm. 4.15]{MR1787949} for
the construction of the stable model structure on the category
of symmetric $T$-spectra.  We will write $\mathcal{SH}$
for its homotopy category, which is the Morel-Voevodsky
$\mathbb A ^1$-stable homotopy category.

Let $\Sigma _T ^{\infty} X_{+}\in
\mathcal{SH}$,
$X\in Sm_k$ denote the infinite suspension of the simplicial presheaf
represented by $X$ with a disjoint base point (written $F_0(X_+)$ in
\cite[p. 506]{MR1787949}).
By \cite[Prop. 4.19]{MR1787949}, $\mathcal {SH}$ is a tensor 
triangulated category with unit $\mathbf 1 =\Sigma _T ^\infty 
\mathrm{Spec}\; k_+$.  We will write $E(1)$ for $E\otimes 
\Sigma _T ^\infty (\Gm)[-1]$, $E\in \mathcal{SH}$ and inductively
$E(n)=(E(n-1))(1)$, $n\geq 0$.  We observe that the functor 
$\mathcal{SH}\rightarrow \mathcal{SH}$,
$E\mapsto E(1)$ is an equivalence of categories \cite[8.10]{MR1860878}, 
\cite[Thm. 4.3.38]{MR2438151};
we will write $E\mapsto E(-1)$ for
its inverse, and inductively $E(-n)=(E(-n+1))(-1)$,
$n>0$.  By convention $E(0)=E$ for $E\in \mathcal{SH}$.

As in the case of $DM$, it follows from
\cite[Thm. 4.5.67]{MR2438151} that $\mathcal{SH}$ is a 
compactly generated triangulated category \eqref{subsec.somefacts}
with compact generators:
\begin{equation}  \label{eq.SHgens}
 \generators _{\mathcal{SH}}=\{ 
 \Sigma _T ^{\infty}X_+(p): X\in Sm_{k}; p \in \mathbb Z\}.
\end{equation}

For $n\in \mathbb Z$, we will write 
$\generators ^{\mathrm{eff}}_{\mathcal{SH}}(n)\subseteq 
\generators _{\mathcal{SH}}$ 
for the set consisting of compact objects of the form:
\begin{equation}  \label{eq.SHeffgenstw}
 \generators ^{\mathrm{eff}}_{\mathcal{SH}}(n)=\{ 
 \Sigma _T ^{\infty}X_+(p): X\in Sm_{k}; p\geq n\}.
\end{equation}

\subsubsection{}  \label{subsubsec.neffSH}
Let $\mathcal{SH}^{\mathrm{eff}} (n)$ be the full triangulated 
subcategory $Loc(\generators
^{\mathrm{eff}} _{\mathcal{SH}}(n))$ of $\mathcal{SH}$ \eqref{subsec.somefacts}, and $\northogonalSH{n}$
be the orthogonal category 
$Loc(\generators^{\mathrm{eff}} _{\mathcal{SH}}(n))^{\perp}$
\eqref{def.orthn}.  Notice that $\mathcal{SH}^{\mathrm{eff}} (n)$ 
is compactly generated with set of generators
$\generators _{\mathcal{SH}}^{\mathrm{eff}}(n)$ 
\cite[Thm. 2.1(2.1.1)]{MR1308405}.

\subsubsection{}  \label{subsec.birattowSH}  

We will consider \eqref{eq.genslitow}, \eqref{eq.genslitow2},
\eqref{A.birtow}, \eqref{def.abut.fil}, \eqref{gen.spec.seq},
\eqref{hocolim.spec.seq} and \eqref{comp.spec.seq} in $\mathcal{SH}$
for the family $\mathcal S=\{ \generators 
_{\mathcal{SH}}^{\mathrm{eff}}(n) \}_{n\in \mathbb Z}$
of subsets of $\generators _{\mathcal{SH}}$ \eqref{eq.SHeffgenstw}. 
These were constructed in \cite{Pelaez:2014}.
	
\section{Orthogonality and Duality} \label{sec.orth-and-dual}

\subsection{} \label{subsec.orth.dual}

Recall that we are working with $\mathbb Z [\frac{1}{p}]$-coefficients
\eqref{def.DM}.  In this section we will consider
$Y\in Sm_k$ connected of dimension $d$, and $s$, 
$t\in \mathbb Z$.

	\begin{prop} \label{prop.rep.orth}
With the notation and conditions of \eqref{subsec.orth.dual}.
Then:
	\begin{align*}
M(Y)(s)[t] \in \northogonal{d+s+1}.
	\end{align*}
(see \eqref{subsubsec.neffDM} and \eqref{def.orthn}).
	\end{prop}
	\begin{proof}
By \cite[2.1.2]{MR3614974} it suffices to show that 
$\Hom _{DM}(M(X)(a)[b],M(Y)(s)[t])=0$, for every $X\in Sm_k$,
$a$, $b\in \mathbb Z$ such that $a\geq d+s+1$.
So, by \cite[Cor. 4.13, Thm. 4.12 and Thm. 5.1]{MR3590347} we may
assume that the base field $k$ is perfect.
Now, if the base field $k$ admits resolution of singularities, it
follows from \cite[Thm. 4.3.7]{MR1764202}  that:
	\begin{align*}
\Hom _{DM}(M(X)(a)[b],M(Y)(s)[t]) \cong 
\Hom _{DM}(M(X)\otimes M^c(Y)(e)[f],\mathbf 1)
	\end{align*}
where $M^c(Y)\in \DMeff$ is the motive of Y with compact supports
 \cite[\S 4.1, Cor. 4.1.6]{MR1764202}, $e=a-s-d$ and $f=b-t-2d$.
 For a perfect base field of positive characteristic, we obtain the
 same conclusion by \cite[Thm. 5.5.14 and Lem. 5.5.6]{kelly}.
 
 Therefore, by \cite[5.1.1]{MR3614974} it suffices to check that
 $M(X)\otimes M^c(Y)(e)[f]\in \DMeff (1)$
 \eqref{subsubsec.neffDM}, which holds by hypothesis:
 $e=a-s-d\geq 1$.
	\end{proof}
	
	\begin{cor}  \label{cor.orth}
With the notation and conditions of \eqref{subsec.orth.dual}.  
Let $E=M(Y)(s)[t]\in DM$. Then:
	\begin{enumerate}
\item \label{cor.orth.a} The natural map \eqref{s.counit-properties}:
	\begin{align*}
\theta _{d+s}^{E}:bc_{\leq d+s}E\rightarrow E
	\end{align*}
is an isomorphism in $DM$.
\item \label{cor.orth.b} For any $A\in DM$, and any map
$f:E\rightarrow A$ in $DM$, there exists a unique lifting
$g:E\rightarrow bc_{\leq d+s}A$ such that the following
diagram commutes in $DM$:
	\begin{align} \label{eq.cor.orth.b}
	\begin{array}{c}
\xymatrix{ & bc_{\leq d+s}A \ar[d]^-{\theta _{d+s}^{A}}\\
	E \ar[r]_-{f} \ar[ur]^-{g}& A}
	\end{array}
	\end{align}
\item \label{cor.orth.c} The map $f$ in \eqref{eq.cor.orth.b} is
zero if and only if the map $g$ in \eqref{eq.cor.orth.b} is zero.
	\end{enumerate}
	\end{cor}
	\begin{proof}
\eqref{cor.orth.a}: This follows directly by combining \eqref{prop.rep.orth}
with \cite[3.2.7]{MR3614974}.

\eqref{cor.orth.b} and \eqref{cor.orth.c} follow from \eqref{prop.rep.orth} and the
universal property of $\theta _{d+s}^A$ \eqref{s.counit-properties}.
	\end{proof}

\section{Convergence} \label{sec.conv}

\subsection{} \label{subsec.cond}
In this section we will consider objects $A$, $B\in DM$,
where $B$ is of the form $B=M(X)(s)[t]$ for 
$X\in Sm_k$ and  $s$, $t\in \mathbb Z$.

	\begin{thm} \label{thm.prel.conv}
With the notation and conditions of \eqref{subsec.cond}.  Then
the spectral sequence
\eqref{hocolim.spec.seq} is strongly convergent 
\cite[Def. 5.2(iii)]{MR1718076}.
	\end{thm}
	\begin{proof}
Since $B=M(X)(s)[t]$ is compact in $DM$
\eqref{sub.gens}, it follows from \cite[Lem. 2.8]{MR1308405} that
	\begin{align*}
\Hom_{DM}(B,\hocolim_{n\rightarrow \infty}bc_{\leq n}A)
\cong colim_{n\rightarrow \infty} \Hom_{DM}(B,bc_{\leq n}A)
	\end{align*}
which implies that the filtration \eqref{def.abut.fil}
$F_\bullet$ on $\Hom_{DM}(B,
\hocolim_{n\rightarrow \infty}bc_{\leq n}A)$ 
is exhaustive.

Now, we observe that $B \in \DMeff (s)$
for every $t\in \mathbb Z$ \eqref{subsubsec.neffDM}
and by construction
$bc_{\leq n}A \in \northogonal{n+1}$ \eqref{def.birat.cover}, 
so we deduce that
$\Hom _{DM} (B, bc_{\leq n}A)=0$ for all
$n\leq s-1$  and every $t\in \mathbb Z$ \eqref{def.orthn}.  
Hence, applying the distinguished triangle
\eqref{dist.triang.ss}  we conclude that $E^1_{p,q}=0$ for $p\leq s-1$.
Then, \cite[Thm. 6.1(a)]{MR1718076} implies that the spectral
sequence is strongly convergent since the differentials are of the
form $d_r:E^r_{p,q}\rightarrow E^r_{p-r,q-r+1}$ (notice that our
notation is homological while Boardman's is cohomological, see
\cite[(12.1) and Thm. 12.2]{MR1718076} for an explicit comparison).
	\end{proof}

	\begin{cor} \label{main.cor}
With the notation and conditions
 of \eqref{subsec.cond}.  Assume that the
canonical map
$c:\hocolim _{n \rightarrow \infty}bc_{\leq n}A \rightarrow A$
\eqref{A.birtow} induces an isomorphism of abelian groups:
	\begin{align*}
c_{\ast}: \Hom _{DM}(B,\hocolim _{n \rightarrow \infty}bc_{\leq n}A)
\stackrel{\cong}{\rightarrow} \Hom _{DM}(B,A).
	\end{align*}
Then the spectral sequence \eqref{gen.spec.seq} is strongly
convergent \cite[Def. 5.2(iii)]{MR1718076}.
	\end{cor}
	\begin{proof}
Follows directly by combining \eqref{comp.spec.seq} with
\eqref{thm.prel.conv}.
	\end{proof}

The following is the main theorem:

\begin{thm} \label{main.thm}
With the notation and conditions of \eqref{subsec.cond}.
Then the canonical map
\[ c:\hocolim _{n \rightarrow \infty}bc_{\leq n}A 
\stackrel{\cong}{\rightarrow} A
\]
is an isomorphism in $DM$.
Hence, the spectral sequence
\eqref{gen.spec.seq} is strongly convergent. 
\end{thm}
\begin{proof}
By \eqref{main.cor} it is enough to show that $c$
is an isomorphism in $DM$.  In order to prove this, 
it suffices to see \eqref{sub.gens}
that for every $a$, $b\in \mathbb Z$ and every
connected $Y\in Sm_k$ the induced map:
\[ c_{\ast}:\Hom_{DM}(M(Y)(a)[b],\hocolim _{n \rightarrow \infty}bc_{\leq n}A)
 \rightarrow \Hom _{DM}(M(Y)(a)[b],A)
\]
is an isomorphism of abelian groups.

First we show that $c_\ast$ is surjective. In
effect given $f:M(Y)(a)[b]\rightarrow A$ we obtain a lifting
by \ref{cor.orth}\eqref{cor.orth.b}:
\[ \xymatrix{ & bc_{\leq d+a}A \ar[d]^-{\theta _{d+a}} \\
	M(Y)(a)[b] \ar[r]_-{f} \ar[ur]& A}
\]
where $d$ is the dimension of $Y$.  Then the surjectivity follows by
\eqref{A.birtow}.

Finally we consider the injectivity.  Let
$f:M(Y)(a)[b]\rightarrow \hocolim _{n \rightarrow \infty}bc_{\leq n}A$ 
such that the composition $c_\ast (f):M(Y)(a)[b]\rightarrow A$
is zero.  Since $M(Y)(a)[b]$ is compact in $DM$ \eqref{sub.gens},
we conclude that \cite[Lem. 2.8]{MR1308405}:
\[ \Hom_{DM}(M(Y)(a)[b],\hocolim _{n \rightarrow \infty}bc_{\leq n}A)
 \cong \colim _{n \rightarrow \infty} \Hom _{DM}(M(Y)(a)[b],bc_{\leq n}A)
\]
Thus we may assume that $f$ factors as:
\[ \xymatrix{ & bc_{\leq n}A \ar[d]\\
	M(Y)(a)[b] \ar[r]_-{f} \ar[ur]^-{f^\prime}
	& \hocolim _{n \rightarrow \infty}bc_{\leq n}A}
\]
for some $n\geq d+a$.  Now, applying again \ref{cor.orth}\eqref{cor.orth.b} 
we may factor $f^\prime$ as follows:
\[ \xymatrix{ & bc_{\leq d+a}A\cong  bc_{\leq d+a}(bc_{\leq n}A)\ar[d] \\
	M(Y)(a)[b] \ar[r]_-{f^\prime} \ar[ur]^-{f^{''}}& bc_{\leq n}A}
\]
where the isomorphism follows from \cite[3.2.6]{MR3614974} since
$d+a\leq n$.
Thus it suffices to show that $f^{''}$ is zero.  But this follows from 
\ref{cor.orth}\eqref{cor.orth.c} since 
\[  \theta ^A_{d+a}\circ f^{''}=c_{\ast}(f)=0
\]
where the first equality follows from \eqref{A.birtow} and
the two commutative triangles above while the
second equality follows by hypothesis.
	\end{proof}
	
	\begin{rmk} \label{rmk.gen.trcat}
We observe that \eqref{thm.prel.conv} and \eqref{main.cor} hold 
for a compactly generated triangulated category $\mathcal T$
with compact generators $\generators$ and any choice
of a family of subsets of $\generators$: 
$\mathcal S=\{ \generators _{n} \} _{n
\in \mathbb Z}$ satisfying the conditions in \eqref{subsec.slcos}
and in addition $\cup _{n\in \mathbb Z}\; \generators _n =\generators$.

However, \eqref{main.thm} does not hold
for a general compactly generated triangulated category as
we will see in the next section.
	\end{rmk}

\section{The $\mathbb A ^1$-stable homotopy category} \label{SH}

\subsection{}
In this section we show that  \eqref{main.thm} does not hold for the 
sphere spectrum in $\mathcal{SH}$.  On the other hand,
we show that \eqref{main.thm} holds for objects in $\northogonalSH{n}$,
$n\in \mathbb Z$.  As a direct consequence we obtain
a spectral sequence converging to the $E_1$-term of
Voevodsky's slice spectral sequence.

\begin{prop} \label{prop.1notconv}
The canonical map $c:\hocolim _{n\rightarrow \infty}bc_{\leq n}\mathbf 1
\rightarrow \mathbf 1$ \eqref{A.birtow}
is not an isomorphism in $\mathcal{SH}$.
\end{prop}
\begin{proof}
We proceed by contradiction, and assume that $c$ is an isomorphism
in $\mathcal{SH}$.  Then, since $\mathbf 1 \in \mathcal{SH}$
is compact we conclude \cite[Lem. 2.8]{MR1308405}:
	\begin{align*}
\xymatrix{
\colim _{n\rightarrow \infty} \Hom 
_{\mathcal{SH}}(\mathbf 1 , bc_{\leq n}\mathbf 1) \cong
\Hom _{\mathcal{SH}}(\mathbf 1 ,\hocolim _{n\rightarrow \infty}
bc_{\leq n}\mathbf 1)
\ar[r]^-{c_{\ast}}_-{\cong}&
\Hom _{\mathcal{SH}}(\mathbf 1, \mathbf 1)}
	\end{align*}
Thus, for some $n\in \mathbb Z$ the identity map for $\mathbf 1$
factors as in the following commutative diagram in
$\mathcal{SH}$:
	\begin{align*}
\xymatrix{& bc_{\leq n}\mathbf 1 \ar[d]^-{\theta_{n}^{\mathbf 1}}\\
\mathbf 1 \ar[r]_-{id} \ar[ur]& \mathbf 1}	
	\end{align*}
Hence, $\mathbf 1 \oplus E \cong bc_{\leq n}\mathbf 1
\in \northogonalSH{n+1}$  \eqref{def.birat.cover} for some
$E\in \mathcal{SH}$.  Since $\northogonalSH{n+1}$ is closed
under direct summands \eqref{def.orthn}, we deduce that
$\mathbf 1 \in \northogonalSH{n+1}$.  But this is a contradiction
since it implies that for every $m\geq n$, the slice functors
of Voevodsky \cite[Thm. 2.2]{MR1977582} vanish $s_m\mathbf 1=0$ for
the sphere spectrum, which is not the case
\cite[Conj. 9]{MR1977582}, \cite[p. 350]{MR3217623}.
\end{proof}

	\begin{rmk}  \label{rmk.notconv}
The argument above shows that
\eqref{prop.1notconv} holds for any compact object $A\in \mathcal{SH}$
such that for every $n\in \mathbb Z$ there exists $m\geq n$ with
$s_m A\neq 0$.
	\end{rmk}
	
	\begin{cor}  \label{cor.notconv}
Consider the spectral sequence \eqref{gen.spec.seq} in $\mathcal{SH}$ 
for $A=\mathbf 1$.  Then the spectral sequence is not strongly
convergent for every $B=\Sigma _T ^{\infty} X_+(s)[t]$, $X\in Sm_k$,
$s$, $t\in \mathbb Z$.
	\end{cor}
\begin{proof}
We proceed by contradiction and assume that the spectral
sequence is strongly convergent for every $B$ as above.
Then combining
\eqref{comp.spec.seq} and \eqref{thm.prel.conv} (which also holds
in $\mathcal{SH}$ \eqref{rmk.gen.trcat}) we conclude that
	\begin{align*}
c_\ast: \Hom _{\mathcal{SH}}(B, \hocolim _{n\rightarrow \infty} bc_{\leq n}
\mathbf 1)\rightarrow \Hom _{\mathcal{SH}}(B, \mathbf 1)
	\end{align*}
is an isomorphism.  But this implies that $c$ is an isomorphism since
$\mathcal{SH}$ is a compactly generated category with generators
$\generators _{\mathcal{SH}}$ \eqref{eq.SHgens}.
\end{proof}

However, the spectral sequence \eqref{gen.spec.seq} is
strongly convergent for a large class of objects in $\mathcal{SH}$:

	\begin{prop}  \label{prop.orth.conv}
Let $A\in \mathcal{SH}$ such that for some $r\in \mathbb Z$,
$A \in \northogonalSH{r}$ \eqref{subsubsec.neffSH}. Then the canonical map
\[ c:\hocolim _{n \rightarrow \infty}bc_{\leq n}A 
\stackrel{\cong}{\rightarrow} A
\]
is an isomorphism in $\mathcal{SH}$.
Hence, the spectral sequence
\eqref{gen.spec.seq} is strongly convergent for every
$B=\Sigma _T ^{\infty} X_+(s)[t]$,
$X\in Sm_k$, $s$, $t\in \mathbb Z$.
	\end{prop}
	\begin{proof}
By \eqref{main.cor} (which holds as well
in $\mathcal{SH}$ \eqref{rmk.gen.trcat})
it is enough to show that $c$
is an isomorphism in $\mathcal{SH}$.

Let $m\geq r$ be an arbitrary integer.
It follows from \eqref{eq.genslitow2} that
$A\in \northogonalSH{m}$, so by the universal
property \eqref{s.counit-properties} we conclude that
$\theta _m^A:bc_{\leq m}A \rightarrow A$ is an isomorphism
in $\mathcal{SH}$ (see \cite[2.3.7]{MR3614974}).  
But this implies that  the canonical map
$c:\hocolim _{n\rightarrow \infty}bc_{\leq n}A\rightarrow A$
\eqref{A.birtow}
is an isomorphism in $\mathcal{SH}$.
	\end{proof}
	
\subsubsection{} \label{subsubsec.Voevslice}
Let $B\in \mathcal{SH}$ with
$B=\Sigma _T ^{\infty} X_+(s)[t]$, $X\in Sm_k$,
$s$, $t\in \mathbb Z$.  Consider Voevodsky's 
slice spectral sequence \cite[\S 7]{MR1977582} for
$G\in \mathcal{SH}$:
	\begin{align*}
E_1^{m,n}=\Hom _{\mathcal{SH}}(B, s_mG [m+n])\Rightarrow
\Hom _{\mathcal{SH}}(B,G)
	\end{align*}
A direct consequence of \eqref{prop.orth.conv} is the fact that
we obtain a spectral sequence which convergences strongly
to the $E_1$-term of Voevodsky's slice spectral sequence
and which is compatible with the differentials
$d_1:E_1^{m,n}\rightarrow E_1^{m+1,n}$:
	
	\begin{cor} \label{cor.Voev.slice}
With the notation and conditions of \eqref{subsubsec.Voevslice}.
Then the spectral sequence \eqref{gen.spec.seq} for
$A=s_mG[m+n]$ converges strongly to the $E_1$-term
of Voevodsky's slice spectral sequence  
$E_1^{m,n}=\Hom _{\mathcal{SH}}(B, s_mG [m+n])$.
In addition, the differential $d_1:E_1^{m,n}\rightarrow E_1^{m+1,n}$
in Voevodsky's slice spectral sequence induces a map
between the spectral sequences:
	\begin{align} \label{eq.orthtoslic}
		\begin{array}{r}
\xymatrix@C=-0.5pt{E^{1}_{p,q}=\Hom_{\mathcal{SH}}(B, 
(bc_{p/p-1}s_mG[m+n])[q-p]) \ar@{=>}[rrr] \ar[d]^-{d_{1\ast}} &&&
	 E_1^{m,n} \ar[d]^-{d_1} \\
	E^{1}_{p,q}=\Hom_{\mathcal{SH}}(B, (bc_{p/p-1}s_{m+1}G[m+1+n])[q-p]) 
	\ar@{=>}[rrr] &&& E_1^{m+1,n}}
		\end{array}
	\end{align}
	\end{cor}
	\begin{proof}
By construction $s_mG \in \northogonalSH{m+1}$
\cite[Thm. 2.2(3)]{MR1977582}.
Thus the strong convergence of \eqref{gen.spec.seq} for
$A=s_mG[m+n]$ follows directly from \eqref{prop.orth.conv}.

We observe that the differential $d_1:d_1:E_1^{m,n}\rightarrow E_1^{m+1,n}$
in the slice spectral sequence is induced by the map
$\partial [m+n]$ in $\mathcal{SH}$
where $\partial$ is the following composition 
\cite[Thm. 2.2(1)]{MR1977582}:
	\begin{align*}
\xymatrix{s_mG \ar[r]^-{\sigma _m}& f_{m+1}G[1]
\ar[r]^{\pi _{m+1}[1]} & s_{m+1}G[1].}
	\end{align*}
Since the tower \eqref{A.birtow} is functorial in $\mathcal{SH}$
we conclude that $\partial [m+n]$ induces
the desired map of spectral sequences \eqref{eq.orthtoslic}.
	\end{proof}


\bibliography{biblio_conv}
\bibliographystyle{abbrv}

\end{document}